\documentclass[10pt,psamsfonts]{amsart}
  
\usepackage{amsmath}
\usepackage{amsthm}
\usepackage{amssymb}
\usepackage{amscd}
\usepackage{amsfonts}
\usepackage{amsbsy}
\usepackage{graphicx}

\newcommand{\R}{\ensuremath{\mathbb{R}}}

\newtheorem {theorem} {Theorem} %[section]

\begin{document}

\title[Dip and Buffering in a fast-slow system associated to Brain Lactate Kinetics]
{Dip and Buffering in a fast-slow system associated to Brain Lactate Kinetics}

\author[M. Lahutte-Auboin, R. Costalat, J.-P. Fran\c{c}oise and R. Guillevin]
{M. Lahutte-Auboin, R. Costalat, J.-P. Fran\c{c}oise and R. Guillevin}

\address{ M. Lahutte-Auboin: L' H\^opital d' instruction des Arm\'ees du Val-de-Gr\^ace, Service de Radiologie, 74 Bd de Port-Royal, 75005 Paris, France\\
R. Costalat: UPMC, UMI 209, UMMISCO, Universit\'e P.-M. Curie, Paris 6, 75005 Paris, France,\\
J.-P. Fran\c{c}oise: Universit\'e P.-M. Curie, Paris 6, Laboratoire Jacques--Louis Lions,
UMR 7598 CNRS, 4 Pl. Jussieu, 16-26, 75252 Paris, France\\
R. Guillevin: CHU de Poitiers, Universit\'e de Poitiers, Laboratoire MPIM, 2 rue de la Mil\`etrie
BP 577, 86021 Poitiers}
\email{Jean-Pierre.Francoise@upmc.fr}

\subjclass{Primary 34C05, 34A34, 34C14.}

\keywords{Slow-fast dynamics, control, averaging}
\date{17/07/2013}
\dedicatory{}

\begin{abstract}
High-dimensional compartmental dynamical systems have been introduced to model brain metabolism. In this article, an approach is proposed to their mathematical analysis. Reductions of these models are obtained by replacing several compartments by a control term. The analysis focuses here on lactate kinetics. A mathematical analysis of the initial dip of lactacte, observed under stimulation, and of the periodic buffering, underlying the response to a repetitive sequence of identical stimuli, can be proposed. This mathematical analysis relies on asymptotics techniques of time multi-scaled dynamical systems such as averaging along slow manifolds.
\end{abstract}

\maketitle

\section{Introduction}
Several mathematical compartmental models have been recently developped to represent the metabolism of neurons and astrocytes \cite{AC1, AC, ACMP1, ACMP2, PBASMCM, PM12}.  Such models involve too many variables to be accessible to mathematical analysis and untill now they were only discussed via numerical simulations. The viewpoint adopted here is to reduce several compartments contributions into a control, keeping very few dynamical variables, and characterize the controls which are compatible with the observations. We first discuss a two-dimensional controlled system and then move to a four-dimensional system which allows to distinguish between neurons and astrocytes. We focus on these two cases in the understanding of two important phenomena observed in experiments. One is the initial dip in the extracellular lactacte concentration , first reported {\it in vivo} by Hu and Wilson \cite{HW} using an enzyme-based lactate microsensor on rat brain hippocampus after an electrical stimulation of the perforant pathway.  The other is a periodic buffering response to the application of a periodic sequence of stimuli, also observed by Hu and Wilson, and further discussed by Aubert-Costalat-Magistretti-Pellerin \cite{ACMP1}, on the basis of numerical simulations of the Aubert-Costalat system. More precisely, the initial dip can be well explained, in the setting of fast-slow dynamics, by the theorem of existence of slow manifold, when the critical manifold is transversally attractive. The periodic buffering is mathematically justified using both averaging theory on a slow manifold and a theorem of Flatto-Levinson \cite{FlaLevin}. To make the article self-contained a short proof of this theorem is included and  the full verifications of the hypothesis required for its application are provided.
\vskip 1pt
Let us consider a (fast-slow)  dynamical system equiped with a slow control $J(t,x)$ and an imput $F(t)$:
\begin{equation}\label{S1}
\begin{array}{l}
\frac{dx}{dt}=\epsilon'[J(t,x)-C(\frac{x}{k+x}-\frac{y}{k'+y})]\\
\epsilon\frac{dy}{dt}= F(t)(L-y)+C(\frac{x}{k+x}-\frac{y}{k'+y}).\\
\end{array}
\end{equation}
The two variables $x$ and $y$ stand for the lactate concentration in an interstitial (extracellular) domain, respectively in a capillary domain. The experimental context refers to Hu and Wilson \cite{HW} where an electrical stimulus is applied and is represented by the impulsive term $F(t)$. In mathematical terms, $F(t)$ is an imput function that we assume continuous piecewise linear: $F(0)=F_0>0, F(t)=F_1 t\in [t_0,t_1]$, $F(t)=F_0, t\in[t_f,T]$. The time scales $\epsilon, \epsilon'$ are assumed to be small enough within the domain of validity of asymptotics methods of fast-slow systems. 
\vskip 1pt
The term $C(\frac{x}{k+x}-\frac{y}{k'+y})$ stands for a cotransport through the brain-blood-boundary. Justifications for such an expression can be found, for instance, in \cite{KS}. The full justification of the model stands on several previous articles \cite{ACMP1}, \cite{G}, \cite{CFLG}, \cite{LGFC} related to physiology and to physiopathology where emerged gradually the pertinence of the time multi-scaled approach. Finally the interaction with a third intracellular compartment (which includes both neurons and astrocytes) is embodied in a control term $J(t,x)$. The aim of this modeling is to characterize the control terms compatible with the observation of the transients (and asymptotics) of the solutions $(x(t), y(t))$. Of course, such a controlled dynamical system could be applied in a wide variety of context (not necessarily lactate kinetics) which would encompass capillary and bbb transport (or transport through cellular membrane). Note finally that $T,k,k'$ are seen as fixed parameters.
\vskip 1pt
In a second part of the article, we further discuss a natural extension of the system where the intracellular compartment splits into neurons and astrocytes and which  includes transports through cell membranes. Following the choice made in \cite{AC}, a direct transport from capillary to intracellular astrocytes is included. The imput $F(t)$ is kept and we add to the control $J_0(t,x,u,v)$, still applied to the intracellular compartment, two other independent controls $J_1(t,x, u,v)$ and $J_2(t, x,u,v )$ (resp.) responsible for the intracellular lactate dynamics inside the neurons (resp.) astrocytes. This yields a $4$-dimensional fast-slow dynamics with slow control:

\begin{equation}\label{S2}
\begin{array}{l}
\frac{dx}{dt}=\epsilon'[J_0(t,x,u,v)+C_1(\frac{u}{k_n+u}-\frac{x}{x+k}) +C_2(\frac{v}{k_a+v}-\frac{x}{x+k})-C(\frac{x}{k+x}-\frac{y}{k'+y})]\\
\frac{du}{dt}=\epsilon'[J_1(t,x,u,v)-C_1(\frac{u}{k_n+u}-\frac{x}{x+k})] \\
\frac{dv}{dt}=\epsilon'[J_2(t,x,u,v)-C_2(\frac{v}{k_a+v}-\frac{x}{x+k})-C_a(\frac{v}{k_a+v}-\frac{y}{k'+y})]\\
\epsilon\frac{dy}{dt}= F(t)(L-y)+C(\frac{x}{k+x}-\frac{y}{k'+y})+C_a(\frac{v}{k_a+v}-\frac{y}{k'+y}).
\end{array}
\end{equation}

The same analysis can be extended to this $4$-dimensional system and further new possibilities offered by this extension are discussed using a quasi-stationary approximation.  

\section{Dip and buffering in the $2$-dimensional system}

\subsection{Qualitative analysis of the system with fixed controls: stationary point and slow manifold}

Recall shortly the local analysis of the $2$-dimensional system obtained by freezing the imput and the control term:

\begin{equation}\label{S0}
\begin{array}{l}
\frac{dx}{dt}=\epsilon'[J-C(\frac{x}{k+x}-\frac{y}{k'+y})]\\
\epsilon\frac{dy}{dt}= F(L-y)+C(\frac{x}{k+x}-\frac{y}{k'+y}).\\
\end{array}
\end{equation}

\vskip 1pt
\begin{theorem}
The system ($\ref{S0}$) displays a unique stationary point which is of node type.
\end{theorem}

\begin{proof}
Solving the system:
\begin{equation}
\begin{array}{l}
0=J-C(\frac{x}{k+x}-\frac{y}{k'+y})\\
0= F(L-y)+C(\frac{x}{k+x}-\frac{y}{k'+y}),\\
\end{array}
\end{equation}
yields
\begin{equation}
y=L+\frac{J}{F}=y_0,
\end{equation}

which is always positive. This displays:
\begin{equation}
x=\frac{k(\frac{J}{C}+\frac{y_0}{k'+y_0})}{1-(\frac{J}{C}+\frac{y_0}{k'+y_0})}=x_0.
\end{equation}
There is, thus, a unique stationary point $(x_0,y_0)$ and we further discuss its nature.
\vskip 1pt
The eigenvalues $\lambda_{\pm}$ of the Jacobian of the vector field solve the equation:
\begin{equation}
(A+\lambda)(\frac{B+F}{\epsilon}+\lambda)-\frac{AB}{\epsilon}=0,
\end{equation}
with
\begin{equation}
A=\frac{kC}{(k+x)^2}, B=\frac{k'C}{(k'+y)^2}.
\end{equation}
The eigenvalues are so that:
\begin{equation}
\begin{array}{l}
\lambda_++\lambda_-=-(A+\frac{B+F}{\epsilon})<0,\\
\lambda_+\lambda_-=AF/\epsilon>0,
\end{array}
\end{equation}
hence the stationary point is stable. Furthermore,
\begin{equation}
\begin{array}{l}
\Delta=(A+\frac{B+F}{\epsilon})^2-4\frac{AF}{\epsilon}=\\
A^2+2A(\frac{B+F}{\epsilon})+(\frac{B+F}{\epsilon})^2-4\frac{AF}{\epsilon}>\\
A^2-2A(\frac{B+F}{\epsilon})+(\frac{B+F}{\epsilon})^2\geq 0,\\
\end{array}
\end{equation}
hence this unique stationary point is a node.
\end{proof}
Next, we use the fast-slow nature of the system to analyse its global phase portrait. The full phase portrait (for any values of $x$ and $y$) requires some attention because the vector field is rational and displays poles. But we can note that the vector field is entrant on the boundary of the "meaningful domain" where $x\geq 0, y\geq 0$ represents concentrations.

\begin{theorem}
For $\epsilon$ small enough, the system $(\ref{S0})$ displays a slow attractive invariant manifold which is a graph (parabola) over the $y$-axis and a globally attractive stationary point in restriction to the positive quadrant. 
\end{theorem}
\begin{proof}
The critical manifold of $(\ref{S0})$ is given by the equation:

\begin{equation}
\Phi(x,y)=F(L-y)+C(\frac{x}{k+x}-\frac{y}{k'+y})=0.
\end{equation}
Given $y=y_0$, there exists a single solution $x=x_0$ to the equation $\Phi(x,y_0)=0$. Indeed, write
\begin{equation}
B(y_0)=\frac{Cy_0}{k'+y_0}-F(L-y_0),
\end{equation}
yields
\begin{equation}
x=x_0:=\frac{kB(y_0)}{C-B(y_0)}.
\end{equation}
Given $x=x_0$, there are two solutions $y$ to the equation $\Phi(x_0,y)=0$ obtained by solving

\begin{equation}
y^2-[F(L-k')-\frac{Ck}{k+x_0}]y-(\frac{k'Cx_0}{k+x_0}+FLk')=0
\end{equation}

Note that for $x_0>0$ the product of the two roots is $-(\frac{k'Cx_0}{k+x_0}+FLk')<0$. Hence the parabola $\Phi(x,y)=0$ intersects the $y$-axis in one positive and one negative value of $y$. This shows that the critical manifold is also a graph over the $x$-axis in restriction to the positive quadrant $x\geq 0, y\geq 0$.
\vskip 1pt
Note now that
\begin{equation}
\Phi'_y(x,y)=-F-\frac{Ck'}{(k'+y)^2}<0,
\end{equation}
hence the critical manifold is attractive. Usual techniques of study of fast-slow systems can be used. In particular, Tikhonov theorem can be applied \cite{T} (For a discussion of the Tikhonov theorem see \cite{v}). In restriction to the positive quadrants, for $\epsilon$ small enough, all orbits of the flow display a fast attraction to a slow manifold which is $\epsilon$-close to the critical manifold. Along this slow manifold, the flow is conjugated to the restriction of the slow flow on the critical manifold and it converges to the stable stationary point.
\end{proof}

\subsection{Geometrical explanation for the dip}

Let us consider now the impulsive dynamics defined in the introduction by the piecewise constant function $F(t)$. For fixed value $F(t)=F_0, F(t)=F_1$, we can apply the preceding results. As $F(t)$ switches from one value to the other, the critical manifold jumps from one level to another level. We can now interpret geometrically the numerical result of the article \cite{ACMP1}. The authors choose a trapezoidal function $F(t)$, which varies linearly from $F(0)=F_0$ to $F(t_i)=(1+\alpha_F)F_0$, then remains constant ($=(1+\alpha_F)F_0)$ between $t_i$ and $t_f$ and then decreases linearly back to $F_0$ between $t_f$ and $t_f+t_i$.
\vskip 1pt
Before the imput, a generic solution of the system approaches rapidly the slow manifold and then slowly moves to the stationary point. After this transient, we can assume that the initial data is the stationary point. Then, the imput is applied. The main effect is that the slow manifold is moved to a lower position and as a consequence, the solution moves quickly toward the slow manifold which is now below. Then it moves slowly to the new stationary point which is slightly displaced on the left. If the control $J(t,x)$ is applied, it pushes the slow nullcline farther left. Then the solution goes farther left as well in pursuit of the stationary point. This initiates a notable dip of the extracellular lactate concentration $x$. Of course, if the control $J(t,x)$ is turn off, the solution goes back to the initial stationary point.
\vskip 1pt
Although, this is not quite the experimental situation observed by Hu and Wilson. Indeed, they observed that the initial dip is followed by an augmentation of the extracellular lactate concentration $x$ before returning back to the equilibria. This motivates \cite{ACMP1} the choice of a plausible control $J(t)$. For some initial time $t'_i$, $J(t)$ is linearly increased from $J_0$ to some level $J_1$ where it remains constant for a while and the system displays the dip. Then $J(t)$ is decreased below the equilibrium level $J_0$ to $J_{-1}$. As a consequence, the slow nullcline is now moved to the right, and the extracellular concentration $x$ increases as observed. After some time where $J(t)$ is kept constant at this lower level, it is then removed back to its equilibrium $J(0)$. 
\vskip 1pt
\subsection{Geometrical analysis of the buffering: frequency locking in a forced fast-slow system}

In the same experimental conditions, Hu and Wilson considered also the feedback to a sequence of repeated stimuli. They observed a buffering of the interstial lactate concentration which tends rather quickly to an asymptotically frequency locking response. Latter, in the article \cite{ACMP1}, this was confirmed by numerical simulations. We show indeed, by averaging theory on a slow manifold, that the impulsive system with a periodic imput $F(t)$ of period $T$ displays, for some large choice of $T$-periodic controls $J(t,x)$, an attractive periodic orbit, in the limit where  $\epsilon'$ and $\epsilon$ are small enough. This is the main theorem involved:
\vskip 1pt
\begin{theorem}
Let us  consider a fast-slow $(p,q)$ system in balanced form:
\begin{equation}\label{S3}
\begin{array}{l}
\frac{dx}{dt}={\epsilon'}f(x,y,t)\\
{\epsilon}\frac{dy}{dt}=g(x,y,t),\\
x=(x_1,...,x_p)\in D\subset \R^p, y=(y_1,...,y_q)\in G\subset \R^q
\end{array}
\end{equation}
where $f,g$ are $T$-periodic in $t$. Assume that
\vskip 1pt
\noindent a- The equation $g(x,y,t)=0$ is solved by $y=\phi(x,t)$ which is an asymptotically stable solution of the equation:
\begin{equation}
\frac{dy}{d\tau}=g(x,y,\tau),
\end{equation}
uniformly in $x\in D$ and $\tau\in\R_+$ (i.e. there exists a $\mu$, uniform in $x, \tau$ such that ${\rm Re}(\Lambda(g_y))\leq -\mu<0$).
\vskip 1pt
Let $A(t)$ be the $T$-periodic matrix defined as 

\begin{equation}
A(t)=f_x(x,\phi(x,t),t)-f_yg_y^{-1}g_x(x,\phi(x,t),t). 
\end{equation}

\noindent b-Assume that the linear equation:
\begin{equation}
\frac{d\xi}{dt}=A(t)\xi,
\end{equation}
displays no other $T$-periodic solution than $\xi=0$.
\vskip 1pt
\noindent c-There exists an isolated solution $x_0$ to the equation:
\begin{equation}
\frac{1}{T}\int_0^T g(x,\phi(t,x))dt=0.
\end{equation}
Then the fast-slow system displays a $T$ periodic solution with initial data $x_0,y_0=\phi(0,x_0)$ up to $O(\epsilon\epsilon')$.
\end{theorem}
This is typically a situation of averaging in a slow manifold (cf. for instance in \cite{Ver}). Indeed the theorem follows from the usual averaging theorem (see for instance \cite{f}, \cite{v}) and a theorem of \cite{FlaLevin}. For sake of completeness we include a proof here which follows the lines of the article \cite{FlaLevin}:
\begin{proof}
1-Assume that the slow limit of the equation:
\begin{equation}
\begin{array}{l}
\frac{dx}{dt}=f(x,y,t)\\
0=g(x,y,t)
\end{array}
\end{equation}
displays a $T$-periodic solution $x=\theta(t), y=\phi(t)$. Change first variable: $\xi=x-\theta(t), \eta=y-\phi(t)$ and get:

\begin{equation}
\begin{array}{l}
\frac{d\xi}{dt}=f_x \xi+f_y \eta+R,\\
{\epsilon}\frac{d\eta}{dt}=g_x\xi+g_y\eta+S.
\end{array}
\end{equation}
Then eliminate the term $g_x \xi$ by changing $\xi$ into $w$ so that $w=y-\phi(t)+U.\xi$, $U=g_y^{-1}g_x$, this yields:
\begin{equation}
\begin{array}{l}
\frac{d\xi}{dt}=(f_x-f_yU)\xi+f_y w+F\\
{\epsilon}\frac{dw}{dt}=g_y w+G.
\end{array}
\end{equation}
2- Consider first equation ${\epsilon}\frac{dw}{dt}=g_y w$. 
The fundamental matrix (or Green function) $\Phi(s,t)$ of
\begin{equation}
{\epsilon}\frac{dw}{dt}=Q(t)w, {\rm Re}(\Lambda(Q(t))\leq -\mu,
\end{equation}
satisfies $\mid \Phi(s,t,\epsilon)\mid\leq K{\rm e}^{-\mu(t-s)/\epsilon}$.
\vskip 1pt
3-Write $A(t)=f_x-f_yU$, assumption is the only $T$-periodic solution of
\begin{equation}
\frac{d\xi}{dt}=A(t)\xi
\end{equation}
is $\xi=0$. Write $\Psi(s,t)$ the Green function of this linear homogeneous equation. Assumption implies that for any ($T$-periodic) $b(t)$ the only $T$-periodic solution of the linear non-homogeneous equation:
\begin{equation}
\frac{d\xi}{dt}=A(t)\xi+b(t),
\end{equation}
writes
\begin{equation}
\xi(t)=\int_0^T \Psi(t,s)b(s)ds.
\end{equation}
\vskip 1pt
The equation
\begin{equation}
\begin{array}{l}
\frac{d\xi}{dt}=(f_x-f_yU)\xi+f_y w+F\\
{\epsilon}\frac{dw}{dt}=g_y w+G.
\end{array}
\end{equation}
displays a $T$-periodic solution if and only if the integral equation:
\begin{equation}
\begin{array}{l}
\xi(t)=\int_0^T \Phi(t,s)[f_y(s)w(s)+H(s,\xi(s),w(s))ds\\
w(t)=\Psi(t,0)[I-\Phi(T,0)]^{-1}\frac{1}{\epsilon}\int_0^T \Phi(T,s)G(s)ds+\frac{1}{\epsilon}\int_0^t \Phi(t,s)G(s)ds,
\end{array}
\end{equation}
has a solution. To prove that this integral equation has a solution, we apply iterative method. The Lipschitz constant might be very high by the presence of $\frac{1}{\epsilon}$ but this factor is killed by the flat estimate
$\mid \Phi(s,t,\epsilon)\mid\leq K{\rm e}^{-\mu(t-s)/\epsilon}$. Note that the convergence to the periodic orbit is very fast.
\vskip 1pt
4-If there exists an isolated solution $x_0$ to the equation:
\begin{equation}
\frac{1}{T}\int_0^T f(x,\phi(t,x),t)dt=0.
\end{equation}
then, by the usual averaging theorem,  the system
\begin{equation}
\begin{array}{l}
\frac{dx}{dt}=f(x,y,t)\\
0=g(x,y,t)
\end{array}
\end{equation}
displays a $T$-periodic solution for $\epsilon'$ small enough and we can apply what precedes.
\end{proof}
To conclude this part, we have to check the conditions of applications of the previous theorem to the system ($\ref{S1}$).
\vskip 1pt
\noindent a- The computations made on paragraph 2 are still valid when $F=F(t)$ is time-dependent. We have seen that the equation

\begin{equation}
g(x,y,t)=F(t)(L-y)+C(\frac{x}{k+x}-\frac{y}{k'+y})=0,
\end{equation}

equivalent to:
\begin{equation}
y^2+y[k'+\frac{Ck}{(k+x)F(t)}-L]-k'L-k'\frac{Cx}{(k+x)F(t)}=0,
\end{equation}

solves into
\begin{equation}
y=\phi(x,t)=\frac{1}{2}[-k'-\frac{Ck}{(k+x)F(t)}+L+\sqrt{\Delta(t)}],
\end{equation}

with
\begin{equation}
\Delta(t)=[k'+\frac{Ck}{(k+x)F(t)}-L]^2+4k'L+4k'\frac{Cx}{(k+x)F(t)}>0.
\end{equation}

The only point to be checked is that

\begin{equation}
g'_y(x,y,t)=-F(t)-C\frac{k'}{(k'+y)^2}\leq -\mu=-{\rm Min}F(t)-\frac{C}{k'}<0.
\end{equation}

\noindent b- The matrix $A(t)$ is in our case a scalar function which writes:

\begin{equation}
A(t)={\epsilon'}[J_x-\frac{kC}{(k+x)^2}\frac{F(t)}{F(t)+\frac{k'C}{(k'+y)^2}}].
\end{equation}

At this point, if we consider the special case $J(t,x)=J(t)$ as in the article \cite{ACMP1}, we can check the condition as $A(t)<0$. 
In general, it looks natural to impose $J_x\leq 0$.

\noindent c- This yields the slow equation in the standard form (for averaging theory):

\begin{equation}
\frac{dx}{dt}=\epsilon' F(x,t,\epsilon)=\epsilon'\{J(t)+F(t)[L(t)-\phi(x,t)]\}+O(\epsilon\epsilon').
\end{equation}

Denote:

\begin{equation}
\begin{array}{l}
\overline{J}(x)=\frac{1}{T}\int_0^T J(t,x)dt,\\
\overline{F}=\frac{1}{T}\int_0^T F(t)dt,
\end{array}
\end{equation}

the averaging of the impulse $F(t)$ and of the periodic control $J(t,x)$. The averaging equation writes:

\begin{equation}
\begin{array}{l}
\frac{dx}{dt}={\epsilon'}\overline{f}(x),\\
\overline{f}(x)=\overline{J}(x)+L\overline{F}-\frac{1}{T}\int_0^T F(t)\phi(t,x)dt.
\end{array}
\end{equation}

This makes the averaging condition $\overline{f}(x)=0$ quite practical to compute for piecewise linear impulse and controls in our case.

\section{Extension to the $4$-dimensional system}

An important issue expected of the large compartmental models (\cite{AC1, AC, ACMP1, ACMP2}) is to represent the subtle metabolic balance between neurons and astrocytes (\cite{PM94, JMW, JAPMW, PM12}) . It seems tempting to push further the reduced model to include two more compartments (intracellular neurons and intracellular astrocytes). We keep focusing on lactates and use the same rules of co-transports.
Following these lines, we find the $4$-dimensional system:

\begin{equation}\label{S4}
\begin{array}{l}
\frac{dx}{dt}=\epsilon'[J_0(t,x,u,v)+C_1(\frac{u}{k_n+u}-\frac{x}{x+k}) +C_2(\frac{v}{k_a+v}-\frac{x}{x+k})-C(\frac{x}{k+x}-\frac{y}{k'+y})]\\
\frac{du}{dt}=\epsilon'[J_1(t,x,u,v)-C_1(\frac{u}{k_n+u}-\frac{x}{x+k})] \\
\frac{dv}{dt}=\epsilon'[J_2(t,x,u,v)-C_2(\frac{v}{k_a+v}-\frac{x}{x+k})-C_a(\frac{v}{k_a+v}-\frac{y}{k'+y})]\\
\epsilon\frac{dy}{dt}= F(t)(L-y)+C(\frac{x}{k+x}-\frac{y}{k'+y})+C_a(\frac{v}{k_a+v}-\frac{y}{k'+y}).
\end{array}
\end{equation}
The extracellular lactate concentration is still denoted $x$, as well as the intracapillar concentration $y$. We denote $u$ the intra-neuron concentration and $v$ the intra-astrocyte concentration.

\subsection{System with fixed controls, stationary point}

We consider first the system obtained by freezing the imput $F(t)$ and the three controls $J_i(t,x,u,v), i=0,1,2$:

\begin{equation}\label{S5}
\begin{array}{l}
\frac{dx}{dt}=\epsilon'[J_0+C_1(\frac{u}{k_n+u}-\frac{x}{x+k}) +C_2(\frac{v}{k_a+v}-\frac{x}{x+k})-C(\frac{x}{k+x}-\frac{y}{k'+y})]\\
\frac{du}{dt}=\epsilon'[J_1-C_1(\frac{u}{k_n+u}-\frac{x}{x+k})] \\
\frac{dv}{dt}=\epsilon'[J_2-C_2(\frac{v}{k_a+v}-\frac{x}{x+k})-C_a(\frac{v}{k_a+v}-\frac{y}{k'+y})]\\
\epsilon\frac{dy}{dt}= F(L-y)+C(\frac{x}{k+x}-\frac{y}{k'+y})+C_a(\frac{v}{k_a+v}-\frac{y}{k'+y}).
\end{array}
\end{equation}

\vskip 1pt
\begin{theorem}
The system ($\ref{S5}$) displays a unique stationary point.
\end{theorem}

\begin{proof}
The equations for the stationary point display:

\begin{equation}
\begin{array}{l}
0=[J_0+C_1(\frac{u}{k_n+u}-\frac{x}{x+k}) +C_2(\frac{v}{k_a+v}-\frac{x}{x+k})-C(\frac{x}{k+x}-\frac{y}{k'+y})]\\
0=[J_1-C_1(\frac{u}{k_n+u}-\frac{x}{x+k})] \\
0=[J_2-C_2(\frac{v}{k_a+v}-\frac{x}{x+k})-C_a(\frac{v}{k_a+v}-\frac{y}{k'+y})]\\
0= F(L-y)+C(\frac{x}{k+x}-\frac{y}{k'+y})+C_a(\frac{v}{k_a+v}-\frac{y}{k'+y}).
\end{array}
\end{equation}
Summing up the $4$ equations yields:

\begin{equation}
y=y_0=L+\frac{J_0+J_1+J_2}{F}.
\end{equation}

Next, the last two equations display:

\begin{equation}
\begin{array}{l}
-(C_2+C_a)\frac{v}{v+k_a}+C_2\frac{x}{x+k}=J_2-C_a\frac{y_0}{k'+y_0}\\
C_a\frac{v}{v+k_a}+C\frac{x}{x+k}=-(J_0+J_1+J_2)+(C+C_a)\frac{y_0}{k'+y_0}.
\end{array}
\end{equation}
This last system can be easily solved into:

\begin{equation}
v=v_0=\frac{k_a[-\frac{CJ_2+C_2(J_0+J_1+J_2}{CC_2+CC_a+C_2C_a}+\frac{y_0}{k'+y_0}]}{1+\frac{CJ_2+C_2(J_0+J_1+J_2}{CC_2+CC_a+C_2C_a}-\frac{y_0}{k'+y_0}},
\end{equation}

\begin{equation}
x=x_0=\frac{k[\frac{C_aJ_2-(C_2+C_a)(J_0+J_1+J_2)}{CC_2+CC_a+C_2C_a}+\frac{y_0}{k'+y_0}]}{1-\frac{C_aJ_2-(C_2+C_a)(J_0+J_1+J_2)}{CC_2+CC_a+C_2C_a}-\frac{y_0}{k'+y_0}}.
\end{equation}

Lastly, the second equation yields:

\begin{equation}
\frac{u}{k_n+u}=\frac{J_1}{F_1}+\frac{x_0}{x_0+k},
\end{equation}

and thus:

\begin{equation}
u=u_0=\frac{k_n[\frac{J_1}{C_1}+\frac{C_aJ_2-(C_2+C_a)(J_0+J_1+J_2)}{CC_2+CC_a+C_2C_a}+\frac{y_0}{k'+y_0}]}{1-[\frac{J_1}{C_1}+\frac{C_aJ_2-(C_2+C_a)(J_0+J_1+J_2)}{CC_2+CC_a+C_2C_a}+\frac{y_0}{k'+y_0}]}.
\end{equation}
\end{proof}
It seems difficult to discuss the question of the stability of this stationary point. But it is quite easy to check that the critical manifold given by the equation:

\begin{equation}
g(y,x,v)= F(L-y)+C(\frac{x}{k+x}-\frac{y}{k'+y})+C_a(\frac{v}{k_a+v}-\frac{y}{k'+y})=0,
\end{equation}

again defines a branch of hypersurface of equation

\begin{equation}
y=\phi(x,v), y\geq 0
\end{equation}

obtained explicitely by solving a quadratic equation. This branch is again attractive because:

\begin{equation}
g'_y(y,v,x)=-F-(C+C_a)\frac{k'}{(k'+x)^2}<0.
\end{equation}
A similar analysis for the dip can be made. It is also possible to use the previous theorem provided that the technical condition b is fullfilled. This yields again to the possibility of a periodic buffering.
\vskip 1pt
The ANLSH (astrocyte-neuron lactate shuttle hypothesis) posed by L. Pellerin and P.J. Magistretti in 1994 (see \cite{PM94, JMW, JAPMW,PM12}) emphasizes that astrocytes act as a syncitium to distribute energy substrates such as lactate to active neurons.  In relation with this astrocyte-neuron metabolic cooperation, it may be interesting to notice that the system ($\ref{S5}$) allows for a simple quasi-stationary analysis. In particular, we can see from the explicit formula obtained above for the equilibrium value $u=u_0$, the influence of the controls term on an increase of $u_0$ which would correspond  to a consumption by neurons of  extracellular lactate.  More precisely:

If $J_0+J_1+J_2$ is decreased then $y_0$ drops.  If 
\begin{equation}
\frac{C_a J_2-(C_2+C_a)(J_0+J_1+J_2)}{CC_2+CC_a+C_2C_a}
\end{equation}
decreases then $x_0$ decreases. Furthermore if $J_1$ is increased then $u_0$ increases  and if 

\begin{equation}
\frac{CJ_2+C_2(J_0+J_1+J_2)}{CC_2+CC_a+C_2C_a}
\end{equation}
increases then $v_0$ decreases. 
\vskip 1pt
 From the mathematical point of view, it seems interesting to see a rather unusual setting of periodic control theory. Indeed previous references of periodic control theory display examples where non-periodic control is used to approach to a given periodic trajectory (like in satellites dynamics see for instance \cite{BFT}, \cite{K}). Here, the periodic orbit is built by the control (or forcing term) $J(t)$. Also there are more mathematical perspectives to develop about the asymptotics related with averaging in a slow manifold.


\begin{thebibliography}{99}

\bibitem{AC1}{\sc A. Aubert, R. Costalat}
{\it A model of the coupling between Brain electrical Activity, Metabolism and Hemodynamics: Application to the interpretation of Functional Neuroimaging}
Neuroimage 17, 1162-1181 (2002)

\bibitem{AC}
 {\sc Aubert A, Costalat R} 
{\it  Interaction between astrocytes and neurons studied using a mathematical model of compartmentalized energy metabolism} 
J Cereb Blood Flow Metab. 25: 1476-1490, (2005)

\bibitem{ACMP1}{\sc A. Aubert, R. Costalat, P.-J. Magistretti and L. Pellerin}
 {\it Brain lactate kinetics: Modeling evidence for neuronal lactate uptake upon activation}
Proc. Nat. Ac. Sci. vol. 102, n°45, 16448-16453, (2005)

\bibitem{ACMP2}{\sc A. Aubert,L. Pellerin, P.-J. Magistretti, R. Costalat}
{\it A coherent neurobiological framework for functional neuroimaging provided by a model integrating compartimentalized energy metabolism}
Proc. Natl. Acad. Sci. USA, 104, 4188-4193, (2007)

\bibitem{PBASMCM}{\sc L. Pellerin {\it et al}}
{\it Activity-Dependent Regulation of Energy Metabolism by Astrocytes: An update}
Glia, 55: 1251-1262 (2007)

\bibitem{BAM}{\sc M. B\'elanger, I. Allaman and PJ Magistretti}
{\it Brain Energy Metabolism: Focus on Astrocyte-Neuron Metabolic Cooperation}
Cell Metabolism (Review) 14, 724-738 (2011)

\bibitem{PM12}{\sc Pellerin L, Magistretti PJ}
{\it Sweet sixteen for ANLS} 
J Cereb Blood Flow Metab  32, 7 , 1152-1166 (2012)


\bibitem{FlaLevin}
{\sc Flatto L and Levinson N}
{\it Periodic solutions of singularly perturbed systems}
J. Rat. Mech. Analysis, 4, 943-950 (1955)

\bibitem{HW}
{\sc Hu Y., Wilson GS}
{\it A temporary local energy pool coupled to neuronal activity: fluctuations of extracellular lactate levels in rat brain monitored with rapid-response enzyme-based sensor}
J. Neurochem. 69: 1494-1490, (1997)

\bibitem{KS}
{\sc  Keener J and Sneyd J}
{\it Mathematical physiology} Series Interdisciplinary Applied Mathematics, {\bf Vol} 8, 
Springer-Verlag, New-York, (second edition) (2009)


\bibitem{G}
{\sc Guillevin R {\it et al}}  
{\it Mathematical modeling of energy metabolism and hemodynamics of WHO grade II gliomas using in vivo MR data.}
 CR Biologies.334: 31-38, (2011)

\bibitem{CFLG}
{\sc R. Costalat {\it et al}}
{\it Mathematical Modeling of Metabolism and Hemodynamics}
Acta Biotheor. 60: 99-107 (2012)

\bibitem{LGFC}
{\sc M. Lahutte-Auboin {\it et al}}
{\it On a minimal model for hemodynamics and metabolism of Lactate: Application to low grade glioma and therapeuthic strategies}
Acta Biotheor. 61, 79-89 (2013)

\bibitem{BFT}
{\sc B. Bonnard, L. Faubourg et E. Tr\'elat}
{\it M\'ecanique c\'eleste et contr\^ole des v\'ehicules spatiaux} 
Math\'ematiques et Applications, Springer Berlin (2005)

\bibitem{K}
{\sc H. K. Khalil}
{\it Non linear Systems} 
Prentice Hall, Upper saddle River third edition) (2002)

\bibitem{f}{\sc JP Fran\c{c}oise}
{\it Oscillations en Biologie: Analyse qualitative et mod\`eles}
Math\'ematiques et Applications,  Springer Berlin, (2005)

\bibitem{T}
{\sc A. N. Tikhonov}
{\it Systems of differential equations containing a small parameter multiplying the derivative}
Mat. Sb. 31, 575-586 (1952)


\bibitem{v}
{\sc Verhulst F}
{\it Methods and Applications of Singular Perturbations, Boundary Layers and Multiple Timescale Dynamics}
Springer-Verlag, NY (2005)

\bibitem{Ver}
{\sc Verhulst F}
{\it Periodic solutions and slow manifolds}
Int. Journal of Bifurcations and Chaos, vol. 17, No 8, 2533-2540 (2007)

\bibitem{PM94}
{\sc Pellerin L and Magistretti PJ} 
{\it Glutamate uptake into astrocytes stimulates aerobic glycolysis: A mechanism coupling neuronal activity to glucose utilization}
Proc Natl Acad Sci USA. 91: 10625-10629, (1994)

\bibitem{JMW}
{\sc R Jolivet, PJ Magistretti and B Weber}
{\it Deciphering neuron-glia compartimentalization in cortical energy metabolism}
Front. Neuroenerg. 1, 4, 1-10, (2009)


\bibitem{JAPMW}
{\sc Jolivet R, Allaman J, Pellerin L, Magistretti PJ and Weber B}
{\it Comment on recent modeling studies of astrocyte-neuron metabolic interactions}
 J Cereb Blood Flow Metab., 30, 1982-1986 (2010)


\end{thebibliography}
\end{document}